\documentclass[a4paper,11pt]{amsart}
 \usepackage[arrow,matrix]{xy}
\usepackage{amsmath,amssymb,amscd,bbm,amsthm,mathrsfs,dsfont,chemarrow}
\usepackage{graphicx}
 \numberwithin{equation}{section}
\newtheorem{theorem}{Theorem}[section]
\newtheorem{lemma}{Lemma}[section]
\newtheorem{corollary}{Corollary}[section]
\newtheorem{proposition}{Proposition}[section]

\theoremstyle{definition}

\theoremstyle{remark}%{\hspace{0.8cm}  \bf Remark}
\theoremstyle{example}
\newtheorem{example}{Example}[section]

\begin{document}\numberwithin{equation}{section}
\title[On Einstein Kropina metrics*]{On Einstein Kropina metrics*}
\author{Xiaoling Zhang and Yi-Bing Shen} \thanks{*
Supported by National Nature Science Foundation in China (No.
11171297)}
 \begin{abstract}  In this paper, a characteristic condition of Einstein Kropina metrics is given. By the characteristic condition, we prove that a non-Riemannian  Kropina metric $F=\frac{\alpha^2}{\beta}$
with constant Killing form $\beta $ on an n-dimensional manifold
$M$, $n\geq 2$, is an Einstein metric if and only if
$\alpha$ is also an Einstein metric.
By using the navigation data $(h,W)$, it is proved that an
n-dimensional ($n\geq2$)  Kropina metric $F=\frac{\alpha^2}{\beta}$
is Einstein if and only if the Riemannian metric $h$ is Einstein and $W$
is a unit Killing vector field with respect to $h$. Moreover, we show that every Einstein Kropina metric
must have vanishing S-curvature, and any conformal map between Einstein Kropina metrics must be homothetic.\par
\end{abstract}
\maketitle

\vspace {1cm}

\section{Introduction}\par
Let $F$ be a Finsler metric on an $n$-dimensional manifold $M$. $F$ is
called an Einstein metric with Einstein scalar $\sigma$ if
\begin{equation}\label{za1}
Ric=\sigma  F^2,
\end{equation}
where  $\sigma=\sigma(x)$ is a scalar function on $M$. In
particular,  $F$ is said to be Ricci constant (resp. Ricci flat) if
$F$ satisfies \eqref{za1} where $\sigma =$const. (resp. $\sigma=0$).
\par

Recently, some progress has been made on Finsler Einstein metrics of
$(\alpha,\beta)$ type. The $(\alpha,\beta)$-metrics form an
important class of Finsler metrics appearing iteratively in
formulating Physics, Mechanics, Seismology, Biology, Control Theory,
etc.(see \cite{anto,raf,ya}). D. Bao and C. Robles have shown that
every Einstein Randers metric of dimension $n(\geq 3)$ is
necessarily Ricci constant. A $3$-dimensional Randers metric is
Einstein if and only if it is of constant flag curvature, see
\cite{bao1}. For every non-Randers $(\alpha,\beta)$-metric
$F=\alpha\phi(s),s=\frac{\beta}{\alpha}$ with a polynomial function
$\phi(s)$ of degree greater than 2, Cheng has proved that it is an
Einstein metric if and only if  it is Ricci-flat( \cite{cheng}).\par

The Kropina metric is an $(\alpha,\beta)$-metric where $\phi
(s)=1/s$, i.e., $F={\alpha^2}/{\beta}$, which was considered by V.K.Kropina firstly(\cite{Kr}).  Such a metric is of physical interest
in the sense that it describes the general dynamical system
represented by a Lagrangian function (cf. \cite{As}), although it
has the singularity. Some recent progress on Kropina metrics has
been made, e.g., see \cite{raf,ya,yo}.

The purpose of this paper is to investigate  Einstein Kropina
metrics $F=\frac{\alpha^2}{\beta}$, for which we shall restrict our
consideration to the domain where $\beta=b_i(x)y^i>0$. By using a
complicated computation, we obtain the characteristic conditions of
Einstein Kropina metrics in Theorem \ref{zza1} and Theorem
\ref{zza2}, which generalize and improve the resuts of \cite{rez}.

For an $(\alpha,\beta)$-metrics, the form $\beta$ is said to be Killing (resp. closed) form if
$r_{ij}=0$\,\,(resp. $s_{ij}=0$). $\beta$ is said to be a
constant Killing form if it is a Killing form and has constant
length with respect to $\alpha$, equivalently $r_{ij}=0,s_i=0$.  And
accordingly, a vector field $W$ in a Riemannian manifold $(M,h)$ is said to be a constant Killing vector field if
it is a Killing vector field  and has constant length with respect to the Riemannian metric $h$.
\par For $(\alpha,\beta)$-metrics with constant Killing form, by using the characteristic condition of Einstein Kropina metrics,  we have the following
theorem.

 \begin{theorem}\label{zza2}
Let $F=\frac{\alpha^2}{\beta}$ be a non-Riemannian  Kropina metric
with constant Killing form $\beta $ on an n-dimensional manifold
$M$, $n\geq 2$. Then $F$ is an Einstein metric if and only if
$\alpha$ is also an Einstein metric. In this case, $
\sigma=\frac{1}{4}\lambda b^2\geq0$, where $\lambda=\lambda(x)$ is
the Einstein scalar of $\alpha$. Moreover, $F$ is Ricci constant
when $n\geq3$.
\end{theorem}

{\bf {Remark.}}\, B. Rezaei, etc., also discussed Einstein Kropina
metrics with constant Killing form. Unfortunately, the computation
and results in \cite{rez} are wrong. Theorem \ref{zza2} is the
corrected version of Theorem 4.6 and Corollary 4.9 of \cite{rez}.

As is well known, a Finsler metric is of Randers type if and only if
it is a solution of the navigation problem on a Riemannian manifold,
see \cite{bao2}. Inspired by this idea, we can prove that there is a
one-to-one correspondence between a Kropina metric and a pair
$(h,W)$, where $h$ is a Riemannian metric and $W$ is a vector field
on $M$ with the length $||W||_{h}=1$. And we call this pair $(h,W)$
the navigation data of the Kropina metric (see Section \ref{Z5} for
details). The new perspective allows us to characterize Einstein
Kropina spaces as follows.

\begin{theorem}\label{zza3}
Let $F=\frac{\alpha^2}{\beta}$ be a non-Riemannian Kropina metric on
an n-dimensional manifold $M$, $n\geq 2$. Assume the pair $(h,W)$ is
it's navigation data. Then $F$ is an Einstein metric if and only if
$h$ is an Einstein metric and $W$ is a unit Killing vector field. In
this case, $\sigma=\delta\geq0 $, where $\delta=\delta (x)$ is the
Einstein scalar of $h$. Moreover, $F$ is Ricci constant for
$n\geq3$.
\end{theorem}

For the S-curvature with respect to the Busemann-Hausdorff volume
form, we have the followings.
\begin{theorem} \label{zza4}Every Einstein Kropina metric
$F=\frac{\alpha^2}{\beta}$ has vanishing $S$-curvature.\end{theorem}

Finally, we discuss conformal rigidity for Einstein Kropina
 metrics.
\begin{theorem}\label{zza5}
Any conformal map between Einstein Kropina spaces must be homothetic.
\end{theorem}

The content of this paper is arranged as follows. In \S \ref{Z3} we
introduce essential curvatures of Finsler metrics, as well as
notations and conventions. And we compute the Ricci curvature of
Kropina metrics. The characterization of Einstein Kropina metrics,
i.e., Theorem \ref{zza1}, is obtained in \S \ref{Z4}. By using it,
we obtain Theorem \ref{zza2}. And in \S \ref{Z5} the navigation
version of Theorem \ref{zza1} (Theorem
 \ref{zza3}) is proved. In \S \ref{Z6} we investigate the $S$-curvature of
Kropina metrics and Theorem 1.3 is proved. In the last Section the
conformal rigidity for Einstein Kropina metrics is given.

\vspace{8mm}

\vspace{8mm}

\section{Ricci curvature of Kropina metrics}\label{Z3}

Let $F$ be a Finsler metric on an $n$-dimensional manifold $M$ and
$G^i$ be the geodesic coefficients of $F$, which are defined by
\begin{equation*}
G^i:=\frac{1}{4}g^{il}\{[F^2]_{x^ky^l}y^k-[F^2]_{ x^l}\}.
\end{equation*}
For any $x\in M$ and $y\in T_xM \backslash\{0\}$%$($T_0M$ denotes the zero section of $T_xM$)
, the Riemann curvature
$\textbf{R}_y:=R^i_{\,\,\,k}\frac{\partial}{\partial x^i}\bigotimes
dx^k$ is defined by
\begin{equation*}
 R^i_{\,\,\,k}:=2\frac{\partial G^i}{\partial x^k}-\frac{\partial^2 G^i}{\partial x^m\partial y^k}y^m
 +2G^m \frac{\partial^2G^i}{\partial
 y^m\partial y^k}
 -\frac{\partial G^i}{\partial y^m}\frac{\partial G^m}{\partial
 y^k}.
\end{equation*}
Ricci curvature is the trace of the Riemann curvature, which is
defined by
\begin{equation*}
Ric:= R^m_{\,\,\, m}.
\end{equation*}

By definition, an $(\alpha,\beta)$-metric on $M$ is expressed in the
form $F=\alpha\phi(s)$, $s=\frac{\beta}{\alpha}$,  where
$\alpha=\sqrt{a_{ij}(x)y^iy^j}$  is a positive definite Riemannian
metric, $\beta=b_i(x)y^ i$ a 1-form. It is known that
$(\alpha,\beta)$-metric with $||\beta_x||_{\alpha}<b_0$ is a Finsler
metric if and only if $\phi=\phi(s)$ is a positive smooth function
on an open interval $(-b_0,b_0)$ satisfying the following condition:
\begin{equation*}
\phi(s)-s\phi'(s)+(b^2-s^2)\phi''(s)>0, ~~~~\forall|s|\leq b<b_0,
\end{equation*}
see \cite{CHEN}.

Let
\begin{equation*}\begin{aligned}
r_{ij}=\frac{1}{2}(b_{i|j}+b_{j|i}),~~~~s_{ij}=\frac{1}{2}(b_{i|j}-b_{j|i}),~~~~~~~~~~~~~~~~~~~~~~~~~~~~~~~~
\end{aligned}\end{equation*}
where $ "|" $ denotes the covariant derivative with respect to the
Levi-Civita connection of $\alpha$. Denote
\begin{equation*}\begin{aligned}
&r^i_{\,\,j}:= a^{ik}r_{kj}, ~~~~~~~~~~~ r_j:=b^ir_{ij}, ~~~~~~~~~~~ r:=r_{ij}b^ib^j=b^jr_j,\\
&s^i_{\,\,j}:= a^{ik}s_{kj}, ~~~~~~~~~~~ s_j:=b^is_{ij},\\
\end{aligned}\end{equation*}
where$(a^{ij}):=(a_{ij})^{-1}$ and $b^i:=a^{ij}b_j$. Denote
$r^i:=a^{ij}r_j, s^i:=a^{ij}s_j$, $r_{i0}:=r_{ij}y^j,
s_{i0}:=s_{ij}y^j$, $r_{00}:=r_{ij}y^iy^j, r_0:=r_iy^i$ and
$s_0:=s_iy^i$.

 Let $G^i$ and $\bar{G^i}$ be the
geodesic coefficients of $F$ and $\alpha$, respectively. Then we
have the following lemma.
\begin{lemma}[\cite{li}] \label{Lemma 1} For an
$(\alpha,\beta)$-metric $F =\alpha\phi(s),$ $s
=\frac{\beta}{\alpha}$, the geodesic coefficients $G^i$ are given by
\begin{equation}\label{zb1}
G^i=\bar{G^i} +\alpha Q s^i_{\,\,0} +\Psi(r_{00}-2\alpha
Qs_0)b^i+\frac{1}{\alpha}\Theta(r_{00}-2\alpha Qs_0)y^i,
\end{equation}
where
\begin{equation*}\begin{aligned}
&Q:=\frac{\phi'}{\phi-s\phi'},\\
&\Psi:=\frac{\phi''}{2[\phi-s\phi'+(b^2-s^2)\phi'']},\\
&\Theta:=\frac{\phi\phi'-s(\phi\phi''+\phi'\phi')}{2\phi[\phi-s\phi'+(b^2-s^2)\phi'']}.\\
\end{aligned}\end{equation*}
\end{lemma}

From now on we consider a special kind of $(\alpha,\beta)$-metrics
which is called Kropina-metric with the form
\begin{equation*}
F=\alpha \phi(s),\qquad \phi(s):=s^{-1},\qquad s=\frac{\alpha}{\beta}.
\end{equation*} Throughout the paper we shall restrict our
consideration to the domain where $\beta=b_i(x)y^i>0$, so that
$s>0$.

Now we get the Ricci curvature of Kropina metric by using Lemma
\ref{Lemma 1}.

\begin{proposition}\label{zzc1}
For the Kropina metric $F=\frac{\alpha^2}{\beta}$, its geodesic
coefficients are:
\begin{equation}\label{zc1}
G^i=\bar{G}^i -\frac{\alpha^2}{2\beta}s^i_{\,\,0}
+\frac{1}{2b^2}(\frac{\alpha^2 }{\beta}s_0+r_{00})b^i
-\frac{1}{b^2}(s_0+\frac{ \beta}{\alpha^2}r_{00})y^i.
\end{equation} \end{proposition}\begin{proof}
By a direct computation, we can get \eqref{zc1} from \eqref{zb1}.
\end{proof}

\begin{proposition}\label{zzc2} For the Kropina metric
$F=\frac{\alpha^2}{\beta}$, the Ricci curvature  of $F$ is given by
\begin{equation}\label{zc2}
Ric=\overline{Ric}+T,
\end{equation}
where $\overline{Ric}$ denotes the Ricci curvature  of $\alpha$, and
\begin{equation}\label{zc3}
\begin{aligned}
T=&-\frac{\alpha^2}{b^4\beta}s_0r -\frac{r}{b^4}r_{00}
+\frac{\alpha^2}{b^2\beta}b^ks_{0|k}
+\frac{1}{b^2}b^kr_{00|k}+\frac{n-2}{b^2}s_{0|0}
+\frac{n-1}{b^2\alpha^2}\beta
r_{00|0}+\frac{1}{b^2}(\frac{\alpha^2}{\beta}s_0
+r_{00})r^k_{\,\,\,k}\\
& -\frac{\alpha^2}{\beta}s^k_{\,\,\,0|k} -\frac{1}{b^2}r_{0|0}
-\frac{2(2n-3)}{b^4}r_0s_0-\frac{n-2}{b^4}s_0^2
-\frac{4(n-1)}{b^4\alpha^2}\beta
r_{00}r_0+\frac{2(n-1)}{b^4\alpha^2}\beta r_{00}s_0\\
& +\frac{3(n-1)}{b^4\alpha^4}\beta^2r_{00}^2
+\frac{2n}{b^2}s^k_{\,\,\,0}r_{0k} +\frac{1}{b^4}r_0^2
-\frac{\alpha^2}{b^2\beta}s^k_{\,\,\,0}r_{k}
+\frac{n-1}{b^2\beta}\alpha^2s^k_{\,\,\,0}s_{k}-\frac{\alpha^4}{2b^2\beta^2}s^ks_{k}
-\frac{\alpha^2}{b^2\beta}s^kr_{0k}\\
& -\frac{\alpha^4}{4\beta^2}s^j_{\,\,\,k}s^k_{\,\,\,j}.
\end{aligned}
\end{equation}\end{proposition}\begin{proof}
Let
\begin{equation*}\begin{aligned}
T^i:=-\frac{\alpha^2}{2\beta}s^i_{\,\,0}
+\frac{1}{2b^2}(\frac{\alpha^2 }{\beta}s_0+r_{00})b^i
-\frac{1}{b^2}(s_0+\frac{ \beta}{\alpha^2}r_{00})y^i,
\end{aligned}\end{equation*}
then
\begin{equation*}
G^i=\bar{G}^i +T^i.
\end{equation*}
Thus the Ricci curvature of $F$ is related to the Ricci curvature of
$\alpha$ by
\begin{equation}\label{zc4}
Ric=\overline{Ric}+2T^k_{\,\,\,|k}-y^jT^k_{\,\,\,.\,k|j}+2T^jT^k_{\,\,\,.\,j\,.\,k}-T^k_{\,\,\,.\,j}T^j_{\,\,\,.\,k},
\end{equation}
where $"|"$ and  $"."$  denote the horizontal covariant derivative
and vertical covariant derivative with respect to the Berwald
connection determined by $\bar{G^i}$  respectively.

Note that\begin{equation*}
\begin{aligned}
\beta_{|k}=r_{0k}+s_{0k},\qquad b^2_{\,\,\,|k}=2(r_k+s_k),\qquad b^i_{\,\,|k}=r^i_{\,\,k}+s^i_{\,\,k}.\\
\end{aligned}
\end{equation*}
 By a direct computation, we get
\begin{equation*}
\begin{aligned}
2T^k_{\,\,\,|k}=&-\frac{2\alpha^2}{b^4\beta}s_0r-\frac{2}{b^4}r_{00}r+(\frac{4}{b^4}-\frac{\alpha^2}{b^2\beta^2})r_0s_0
+(\frac{4}{b^4}+\frac{\alpha^2}{b^2\beta^2})s_0^2+\frac{4\beta}{b^4\alpha^2}r_{00}r_0~~~~~~~~~~~~~~~~~~~~~
~~~~~~~~~~~~~~~~~~~~~~~~\\
&+\frac{4\beta}{b^4\alpha^2}r_{00}s_0
+\frac{\alpha^2}{b^2\beta}b^ks_{0|k}
+\frac{1}{b^2}b^kr_{00|k}-\frac{2}{b^2}s_{0|0}
-\frac{2\beta}{b^2\alpha^2}r_{00|0}
\\
&+\frac{1}{b^2}(\frac{\alpha^2}{\beta}s_0
+r_{00})r^k_{\,\,\,k}-\frac{2}{b^2\alpha^2}r_{00}^2
+\frac{\alpha^2}{\beta^2}s^k_{\,\,\,0}(r_{0k}+s_{0k})-\frac{\alpha^2}{\beta}s^k_{\,\,\,0|k},\\
%%%%
-y^jT^k_{\,\,\,.\,k|j}=&\frac{2}{b^4}r_0^2-\frac{2(n-1)}{b^4}r_{0}s_0-\frac{1}{b^2}r_{0|0}
-\frac{2(n+1)\beta}{b^4\alpha^2}r_{00}r_0-\frac{2(n+1)\beta}{b^4\alpha^2}r_{00}s_0
~~~~~~~~~~~~~~~~~~~~~~~~~~~~~~~~~~~~~~~~~~~~~~~~~~~~~~~~~~~~~~~~~~~~~~\\
&+\frac{n+1}{b^2\alpha^2}r_{00}^2+\frac{n+1}{b^2\alpha^2}\beta
r_{00|0}-\frac{2n}{b^4}s_0^2+\frac{n}{b^2}s_{0|0},\\
%%%%
2T^jT^k_{\,\,\,.\,j\,.\,k}
=&\frac{1}{b^4}(\frac{\alpha^2}{\beta}s_0+r_{00})r
-\frac{2(n+2)\beta}{b^4\alpha^2}(\frac{\alpha^2}{\beta}s_0+r_{00})
r_0 -\frac{\alpha^2}{b^2\beta}s^k_{\,\,\,0}r_k
+\frac{2n}{b^4}s_0^2~~~~~~~~~~~~~~~~~~~~~~~~~~~~~~~~~~~~~~~~~~~~~~~~~~~~~~~~~~~~~~~~~~~~~~\\
& +\frac{2(3n+2)\beta}{b^4\alpha^2}r_{00}s_0
 +\{\frac{4(n+1)\beta^2}{b^4\alpha^4}-\frac{n+1}{b^2\alpha^2}\}r_{00}^2
+\frac{n\alpha^2}{b^2\beta}s^k_{\,\,\,0}s_k
+\frac{2(n+1)}{b^2}s^k_{\,\,\,0}r_{0k},
\end{aligned}
\end{equation*}
\begin{equation*}
\begin{aligned}
T^k_{\,\,\,.\,j}T^j_{\,\,.\,k}
&=(\frac{\alpha^2}{b^2\beta^2}+\frac{n+2}{b^4})s_0^2
+\frac{\alpha^2}{b^2\beta}s^k_{\,\,\,0}s_k
+(\frac{2}{b^2}+\frac{\alpha^2}{\beta^2})r_{0k}s^k_{\,\,\,0}
+\frac{\alpha^2}{\beta^2}s_{0k}s^k_{\,\,\,0}~~~~~~~~~~~~~~~~~~~~~~~~~~~~~~~~~~~~~~~~~~~~~~~~~~~~~~~~~~~~~~~~~~~~~~
\\
&+\frac{2(n+4)}{b^4\alpha^2}\beta r_{00}s_0
-(\frac{\alpha^2}{b^2\beta^2}+\frac{4}{b^4})r_{0}s_0
+\frac{\alpha^4}{2b^2\beta^2}s^ks_k +\frac{1}{b^4}r_0^2
 -\frac{6\beta}{b^4\alpha^2}r_{00}r_0\\
 &
+\frac{\alpha^2}{b^2\beta}r_{0k}s^k+(-\frac{2}{b^2\alpha^2}+\frac{(n+7)\beta^2}{b^4\alpha^4})r_{00}^2
+\frac{\alpha^4}{4\beta^2}s^i_{\,\,\,j}s^j_{\,\,\,i}.
\end{aligned}
\end{equation*}
Plugging all of these four terms into \eqref{zc4}, we obtain
\eqref{zc2}. This completes the proof.\end{proof}

{\bf {Remark.}}\, For Riemann curvature and the Ricci curvature of
$(\alpha,\beta)$-metrics, L. Zhou gave some formulas in \cite{zhou}.
However, Cheng has corrected some errors of his formulas in
\cite{cheng}. To avoid making such mistakes, we use the definitions
of Riemann curvature and Ricci curvatures to compute it.

From now on, $"|"$ and  $"."$  denote the horizontal covariant
derivative and vertical covariant derivative with respect to the
Berwald connection determined by $\bar{G^i}$, respectively.

\vspace{8mm}
\section{Equivalent equations of Einstein Kropina metrics}\label{Z4}

The following lemma is necessary for the proof of theorems.

\begin{lemma}\label{zzd1}For $(\alpha,\beta)$-metrics with $r_{00}=c(x)\alpha^2$,
if $\alpha$ is an Einstein, i.e.,
$\overline{Ric}=\lambda(x)\alpha^2$ for some function
$\lambda=\lambda(x)$, then the followings hold
\begin{equation*}\begin{cases}
\begin{aligned}
s^i_{\,\,\,0|i}&=(n-1)c_0+\lambda \beta,\\
b^ks^i_{\,\,\,k|i}&=(n-1)b^kc_k+\lambda b^2,\\
0&=(n-1)b^kc_k+\lambda
b^2+s^k_{\,\,\,|k}+s^k_{\,\,\,j}s^j_{\,\,\,k},\end{aligned}\end{cases}
\end{equation*}
where $c_k:=\frac{\partial c}{\partial x^k}$ and
$c_0:=c_ky^k$.\end{lemma}
\begin{proof}
Let $\beta$ satisfy $r_{00}=c(x)\alpha^2$. Then
\begin{equation}\label{zd1}\begin{cases}
\begin{aligned}
b^js^k_{\,\,\,j|i}&=(b^js^k_{\,\,\,j})_{\,|i}-b^j_{\,\,\,|i}s^k_{\,\,\,j}
=-s^k_{\,\,\,|i}-(r^j_{\,\,\,i}+s^j_{\,\,\,i})s^k_{\,\,\,j}=-s^k_{\,\,\,|i}-cs^k_{\,\,\,i}-s^k_{\,\,\,j}s^j_{\,\,\,i},\\
b^js^k_{\,\,\,j|k}&=-s^k_{\,\,\,|k}-s^k_{\,\,\,j}s^j_{\,\,\,k}.
\end{aligned}\end{cases}
\end{equation}

Assume that $\alpha$ is an Einstein metric with Einstein scalar
$\lambda(x)$. Since $\alpha$ is a Riemann metric, we have the Ricci
identity, i.e., $ b_{j|k|l}-b_{j|l|k}=b^s\bar{R}_{jskl}$, where
$\bar{R}_{jskl}$ denotes the Riemann curvature of $\alpha$.
Contracting both sides of it with $a^{jl}$, we get
\begin{equation*}
\begin{aligned}
a^{jl}(b_{j|k|l}-b_{j|l|k})&=b^l_{\,\,\,|k|l}-b^l_{\,\,\,|l|k}=(r^l_{\,\,\,k}+s^l_{\,\,\,k})_{|l}
-(r^l_{\,\,\,l}+s^l_{\,\,\,l})_{|k}=-(n-1)c_k+s^l_{\,\,\,k|l}\\
&=b^sa^{jl}\bar{R}_{jskl}=b^s\overline{R}_{sk}=\lambda b^s
a_{sk}=\lambda b_k,
\end{aligned}
\end{equation*}
that is
\begin{equation}\label{zd2}
s^l_{\,\,\,k|l}=(n-1)c_k+\lambda b_k.
\end{equation}
This is equivalent to the following identity
\begin{equation*}
s^k_{\,\,\,0|k}=(n-1)c_0+\lambda \beta.
\end{equation*}
Contracting \eqref{zd2} with $b^k$, we get
$b^ks^l_{\,\,\,k|l}=(n-1)b^kc_k+\lambda b^2.$ Comparing it with the
second equation of \eqref{zd1}, we obtain that
\begin{equation*}
0=(n-1)b^kc_k+\lambda b^2+s^k_{\,\,\,|k}+s^k_{\,\,\,j}s^j_{\,\,\,k}.
\end{equation*}This completes the proof.
\end{proof}

Using Proposition \ref{zzc2} and Lemma \ref{zzd1}, we can obtain the
necessary and sufficient conditions for Kropina metrics to be
Einstein metrics.

\begin{theorem}\label{zza1}
Let $F=\frac{\alpha^2}{\beta}$ be the non-Riemann Kropina metric on
an $n$-dimensional manifold $M$. \par
 1) For $n=2$, $F$ is an Einstein metric
if and only if there exist scalar functions $c=c(x),
~\lambda=\lambda(x)$ on $M$ such that $\alpha$ and $\beta$ satisfy
the following equations
\begin{equation}\label{za2}\begin{cases}\begin{aligned}
r_{00}&=c\alpha^2,\\
\overline{Ric}&=\lambda\alpha^2,\\
0&=\lambda~b^2\beta-c s_0+ b^kc_k\beta+
b^ks_{0|k}-b^2 s^k_{\,\,\,0|k}+s^k_{\,\,\,0}s_k.\\
\end{aligned}\end{cases}\end{equation}\par
2) For $n\geq3$, $F$ is an Einstein metric if and only if there
exist scalar functions $c=c(x),~f=f(x)$ on $M$ such that $\alpha$
and $\beta$ satisfy the following equations
\begin{equation}\label{za3}\begin{cases}\begin{aligned}
r_{00}&=c\alpha^2,\\
f\alpha^2&=\overline{Ric}~b^4 +(n-2)\{b^2s_{0|0}
+b^2c_0\beta-2c\beta s_0
-s_0^2-c^2\beta^2\},\\
0&=\{(n-2)s_ks^k-b^2s^k_{\,\,\,|k}-b^2s^i_{\,\,j}s^j_{\,\,i}\}\beta+(n-3)b^2c
s_0 +b^2 b^ks_{0|k}-b^4 s^k_{\,\,\,0|k}+(n-1)b^2 s^k_{\,\,\,0}s_k,\\
\end{aligned}\end{cases}\end{equation}
 where
\begin{equation}\label{za4}
f=-(n-2)b^2c^2-b^2 b^kc_k+(n-2)s^ks_k-b^2
s^k_{\,\,\,|k}-b^2s^i_{\,\,j}s^j_{\,\,i}.
\end{equation}\par
In this case, $
\sigma=-\frac{1}{2b^2}s^ks_k-\frac{1}{4}s^i_{\,\,j}s^j_{\,\,i}$ for
$n\geq2$.
\end{theorem}
\begin{proof} Let $F=\frac{\alpha^2}{\beta}$ be an Einstein metric with
Einstein scalar $\sigma(x)$. Multiplying both sides of \eqref{zc2}
by $b^4\alpha^4\beta^2$ to remove the denominators, we provide the
criterion for the Kropina metric to be an Einstein metric as follows
\begin{equation}\label{zd3}
\begin{aligned}
0=&3(n-1)\beta^4r_{00}^2
+(n-1)\{b^2r_{00|0}-4r_{00}r_0+2r_{00}s_0\}\beta^3\alpha^2\\
+&\{\overline{Ric}~b^4-r_{00}r+b^2b^kr_{00|k}+(n-2)b^2s_{0|0}
+b^2r_{00}r^k_{\,\,k}\\
&-b^2r_{0|0}-(4n-6)r_0s_{0}-
(n-2)s_{0}^2+2nb^2r_{0k}s^k_{\,\,0}+r_0^2\}\beta^2\alpha^4\\
+&\{- s_{0}r+b^2 b^ks_{0|k}+b^2 s_{0}r^k_{\,\,\,k} -b^4
s^k_{\,\,\,0|k}-b^2
s^k_{\,\,\,0}r_{k}+(n-1)b^2 s^k_{\,\,\,0}s_{k}-b^2 r_{0k}s^k\}\beta\alpha^6\\
 -&b^2\{\frac{1}{2}s^ks_k+\frac{b^2}{4}s^i_{\,\,j}s^j_{\,\,i}+\sigma(x)b^2\}\alpha^8.\\
\end{aligned}
\end{equation}

The above equation shows that $\alpha^2$ divides
$3(n-1)\beta^4r_{00}^2$.
 Since $\alpha^2$ is irreducible and $\beta^5$ can factor into
linear terms, we have that
 $\alpha^2$ divides $r_{00}^2$. Thus there exists a function $c(x)$ such that
\begin{equation}\label{zd4}
 r_{00}=c(x)\alpha^2,
\end{equation}
which means that  $\beta$ is a conformal form with respect to
$\alpha$.

By \eqref{zd4}, it is easy to get
\begin{equation}\label{zd5}\begin{cases}
\begin{aligned}
&r_{00}=c\alpha^2,~~~~~~r_{ij}=ca_{ij},~~~~~r_{0i}=cy_{i},~~~r_{i}=cb_{i},~~~~~~r=cb^2,~~~r^i_{\,\,\,j}=c\delta^i_{\,\,\,j},\\
&r_{0k}s^k_{\,\,\,0}=0,~~~~~r_{0k}s^k=cs_0,~~~~r_{0}=c\beta,~~~s^k_{\,\,\,0}r_{k}=cs_{0},\\
&r_{00|k}=c_k\alpha^2,~~~r_{00|0}=c_0\alpha^2,~~~r^k_{\,\,\,k}=nc,~~~r_{0|0}=c_0\beta+c^2\alpha^2,\\
\end{aligned}\end{cases}
\end{equation}
where $y_i:= a_{ij}y^j$.\par Substituting all of these into
\eqref{zd3} and dividing both sides by common factor  $\alpha^4$, we
obtain
\begin{equation}\label{zd6}
\begin{aligned}
0&=\overline{Ric}~b^4\beta^2 +(n-2)\{b^2s_{0|0}
+b^2c_0\beta-2c\beta s_0 -s_0^2-c^2\beta^2\}\beta^2\\
&+b^2\{(n-3)c s_0+(n-2)c^2\beta+ b^kc_k\beta+
b^ks_{0|k}-b^2 s^k_{\,\,\,0|k}\\
&+(n-1)s^k_{\,\,\,0}s_k\}\beta\alpha^2-b^2\{\frac{1}{2}s^ks_k+\frac{b^2}{4}s^i_{\,\,j}s^j_{\,\,i}+\sigma b^2\}\alpha^4.\\
\end{aligned}
\end{equation}

\vspace{4mm}

\textbf{Case I:} n=2. \eqref{zd6} can be simplified as
\begin{equation}\label{zd7}
\begin{aligned}
0=&\overline{Ric}~b^2\beta^2+\{-c s_0+ b^kc_k\beta+
b^ks_{0|k}-b^2 s^k_{\,\,\,0|k}+ s^k_{\,\,\,0}s_k\}\beta\alpha^2\\
&-\{\frac{1}{2}s^ks_k+\frac{b^2}{4}s^i_{\,\,j}s^j_{\,\,i}+\sigma b^2\}\alpha^4.\\
\end{aligned}
\end{equation}
Thus there exists some function $\lambda=\lambda(x)$ such that
\begin{equation}\label{zd8}
\overline{Ric}=\lambda\alpha^2,
\end{equation}
 i.e., $\alpha$ is an Einstein metric.

 We plug \eqref{zd8} into \eqref{zd7}. Then  \eqref{zd7} is equivalent to
\begin{equation}\label{zd9}\begin{cases}
\begin{aligned}
\eta&=\lambda~b^2\beta-c s_0+ b^kc_k\beta+
b^ks_{0|k}-b^2 s^k_{\,\,\,0|k}+s^k_{\,\,\,0}s_k,\\
0&=\beta\eta-\{\frac{1}{2}s^ks_k+\frac{b^2}{4}s^i_{\,\,j}s^j_{\,\,i}+\sigma b^2\}\alpha^2.\\
\end{aligned}\end{cases}
\end{equation}

From the second equation of \eqref{zd9}, we know there exists some
function $f=f(x)$ such that
\begin{equation}\label{zd10}
\beta\eta=f\alpha^2,\\
\end{equation}
where
$f=\frac{1}{2}s^ks_k+\frac{b^2}{4}s^i_{\,\,j}s^j_{\,\,i}+\sigma
b^2.$

Now we consider \eqref{zd10} into two cases: 1) If $\eta=t\beta$ for
some function $t=t(x)$ on $M$, then $t b_ib_j=fa_{ij}$. By the
theory of matrix rank, we know that $t=f=0$. So $\eta=0$; 2) If
$\eta\neq t\beta$ for any function $t=t(x)$ on $M$, then we just
choose the suitable direction $y$, such that $\eta(y)=0$. For the
positive definiteness of $\alpha$, $\alpha(y)\neq 0$, so we get
$f=0$. All in all, $f=0$ and $\eta=0.$

Thus \eqref{zd9} is equivalent to
\begin{equation}\label{zd11}\begin{cases}
\begin{aligned}
0&=\lambda~b^2\beta-c s_0+ b^kc_k\beta+
b^ks_{0|k}-b^2 s^k_{\,\,\,0|k}+s^k_{\,\,\,0}s_k,\\
\sigma&=-\frac{1}{2b^2}s^ks_k-\frac{1}{4}s^i_{\,\,j}s^j_{\,\,i}.\\
\end{aligned}\end{cases}\end{equation}

Conversely, if \eqref{za2} holds, putting them into \eqref{zc2}
yields $Ric=\sigma F^2$, where $\sigma$ is given by the second
equation of \eqref{zd11}. Thus $F$ is an Einstein metric.

\vspace{4mm}

\textbf{Case II:} $n\geq3$. From \eqref{zd6}, we know there exists
some function $f=f(x)$ such that
\begin{equation}\label{zd12}
\overline{Ric}~b^4 +(n-2)\{b^2s_{0|0}
+b^2c_0\beta-2c\beta s_0 -s_0^2-c^2\beta^2\}=f \alpha^2.\\
\end{equation}
Then \eqref{zd6} can be simplified as
\begin{equation}\label{zd13}
\begin{aligned}
0=&\beta\{(n-3)b^2c s_0+(n-2)b^2c^2\beta+b^2 b^kc_k\beta+b^2
b^ks_{0|k}-b^4 s^k_{\,\,\,0|k}+(n-1)b^2 s^k_{\,\,\,0}s_k\\
&+f\beta\}-b^2\{\frac{1}{2}s^ks_k+\frac{b^2}{4}s^i_{\,\,j}s^j_{\,\,i}+\sigma b^2\}\alpha^2.\\
\end{aligned}
\end{equation}

Since $\alpha^2$ can't be  divided by $\beta$, we see that
\eqref{zd13} is equivalent to the following equations
\begin{equation}\label{zd14}\begin{cases}\begin{aligned}
0&=(n-3)b^2c s_0+(n-2)b^2c^2\beta+b^2b^kc_k\beta +b^2 b^ks_{0|k}-b^4
s^k_{\,\,\,0|k}\\
&+(n-1)b^2 s^k_{\,\,\,0}s_k+f\beta,\\
0&=\frac{1}{2}s^ks_k+\frac{b^2}{4}s^i_{\,\,j}s^j_{\,\,i}+b^2\sigma.
\end{aligned}\end{cases}\end{equation}

 Firstly, differentiating
both sides of the first equation of \eqref{zd14} with respect to
$y^i$ yields
\begin{equation}\label{zd15}
0=(n-3)b^2c s_i+(n-2)b^2c^2b_i+b^2 b^kc_k b_i+b^2 b^ks_{i|k}-b^4
s^k_{\,\,\,i|k}+(n-1)b^2 s^k_{\,\,\,i}s_k+fb_i.
\end{equation}
Contracting \eqref{zd15} with $b^i$ gives
\begin{equation}\label{zd16}
0=(n-2)b^4c^2+b^4 b^kc_k-(n-2)b^2 s^ks_k+b^4
s^k_{\,\,\,|k}+b^4s^i_{\,\,j}s^j_{\,\,i}+b^2f.
\end{equation}
Removing the factor $b^2$ from\eqref{zd16}, we obtain
\begin{equation}\label{zd17}
f=-(n-2)b^2c^2-b^2 b^kc_k+(n-2)s^ks_k-b^2
s^k_{\,\,\,|k}-b^2s^i_{\,\,j}s^j_{\,\,i}.
\end{equation}
Plugging \eqref{zd17} into the first equation of \eqref{zd14} yields
\begin{equation*}
0=\{(n-2)s_ks^k-b^2s^k_{\,\,\,|k}-b^2s^i_{\,\,j}s^j_{\,\,i}\}\beta+(n-3)b^2c
s_0 +b^2 b^ks_{0|k}-b^4 s^k_{\,\,\,0|k}+(n-1)b^2 s^k_{\,\,\,0}s_k.
\end{equation*}

 Secondly, by the second
equation of \eqref{zd14}, we obtain the Einstein scalar
\begin{equation}\label{zd18}
\sigma=-\frac{1}{2b^2}s^ks_k-\frac{1}{4}s^i_{\,\,j}s^j_{\,\,i}.
\end{equation}

Conversely, suppose \eqref{za3} and \eqref{za4} hold. Plugging them
into \eqref{zc2}, we conclude that $F$ is an Einstein metric with
Einstein scalar $\sigma$, which is given by \eqref{zd18}. It
completes the proof of Theorem \ref{zza1}. \end{proof}
 \vskip 5mm

By Theorem \ref{zza1}, we can obtain Theorem \ref{zza2}, that is
\vspace{4mm}
\begin{theorem}\label{zzd2}
Let $F=\frac{\alpha^2}{\beta}$ be a non-Riemannian  Kropina metric
with constant Killing form $\beta $ on an n-dimensional manifold
$M$, $n\geq 2$. Then $F$ is an Einstein metric if and only if
$\alpha$ is also an Einstein metric. In this case, $
\sigma=\frac{1}{4}\lambda b^2\geq0$, where $\lambda=\lambda(x)$ is
the Einstein scalar of $\alpha$. Moreover, $F$ is Ricci constant for
$n\geq3$.
\end{theorem}
\begin{proof} Assume that $F$ is an Einstein metric. Substituting  $r_{ij}=0$ and $ s_i=0$ into \eqref{zd6} and
removing the factor $b^4$, we get
\begin{equation}\label{zd19}
0=\overline{Ric}\beta^2- s^k_{\,\,\,0|k}\beta\alpha^2
-\{\frac{1}{4}s^i_{\,\,j}s^j_{\,\,i}+\sigma\}\alpha^4.
\end{equation}
Thus $\overline{Ric}$ is divisible by $\alpha^2$, i.e., there exists
a function $\lambda(x)$ such that
\begin{equation}\label{zd20}
 \overline{Ric}=\lambda\alpha^2.
\end{equation}

Putting \eqref{zd20} into \eqref{zd19} and dividing the common
factor $\alpha^2$, we conclude that
\begin{equation}\label{zd21}
0=\{\lambda\beta-s^k_{\,\,\,0|k}\}\beta
-\{\frac{1}{4}s^i_{\,\,j}s^j_{\,\,i}+\sigma\}\alpha^2.
\end{equation}
By Lemma \ref{zzd1}, we have $ s^k_{\,\,\,0|k}=\lambda \beta,\,
b^ks^i_{\,\,\,k|i}=\lambda b^2=-s^i_{\,\,j}s^j_{\,\,i}$. Thus
\eqref{zd21}  is equivalent to
\begin{equation}\label{zd22}
\sigma=-\frac{1}{4}s^i_{\,\,j}s^j_{\,\,i}=\frac{1}{4}\lambda
b^2.\end{equation}

For $\lambda
b^2=b^ks^l_{\,\,\,k|l}=-s^i_{\,\,j}s^j_{\,\,i}=\parallel
s_{ij}\parallel^2_{\alpha}\geq 0$, $\lambda$ is non negative. Thus
$\sigma=\frac{1}{4}\lambda b^2\geq0$.\par

Conversely, assume $r_{ij}=s_i=0$ and $\alpha$ is an Einstein
metric, i.e., $\overline{Ric}=\lambda(x)\alpha^2$. Then we have $
s^k_{\,\,\,0|k}=\lambda \beta,\, b^ks^i_{\,\,\,k|i}=\lambda
b^2=-s^i_{\,\,j}s^j_{\,\,i}$ by Lemma \ref{zzd1}. Putting all of
these and $r_{ij}=0,s_i=0$  into \eqref{zc2}, we obtain
$0=Ric-\sigma F^2$, where $\sigma=\frac{1}{4}\lambda b^2$. Hence $F$
is an Einstein metric. It completes the proof of Theorem
\ref{zzd2}.\end{proof}

\vspace{4mm}
\begin{corollary} \label{zzd3}Let $F=\frac{\alpha^2}{\beta}$ be a non-Riemannian  Kropina metric with $s_i=0$
on an n-dimensional manifold $M$, $n\geq 3$. If $F$ and $\alpha$ are
both Einstein metrics, then one of the followings holds\par 1)
$\beta$ is a constant Killing form. In this case,
$\sigma=\frac{1}{4}\lambda b^2\geq0$, where $\lambda=\lambda(x)$
denotes the Einstein scalar of $\alpha$.\par
 2) $\beta $ is closed.  In this case,
$\sigma=0$, i.e., $F$ is Ricci flat.
\end{corollary}
\begin{proof}  Let $s_i$=0. Assume that Einstein scalars of $\alpha$ and $F$
are $\lambda$ and $\sigma$ respectively, i.e., $
\overline{Ric}=\lambda(x)\alpha^2$ and $Ric=\sigma(x)F^2$.

By Theorem \ref{zza1}, that $F$ is an Einstein metric with $s_i=0$
is equivalent to
\begin{equation}\label{zd24}\begin{cases}\begin{aligned}
r_{00}&=c\alpha^2,\\
f\alpha^2&=\lambda b^4\alpha^2+(n-2)\{b^2c_0-c^2\beta\}\beta,\\
0&=s^i_{\,\,j}s^j_{\,\,i}\beta+b^2 s^k_{\,\,\,0|k},\\
f&=-(n-2)b^2c^2-b^2 b^kc_k-b^2s^i_{\,\,j}s^j_{\,\,i}.
\end{aligned}\end{cases}\end{equation}
In this case, $\sigma=-\frac{1}{4}s^i_{\,\,j}s^j_{\,\,i}$.

For the same reason in discussing \eqref{zd10}, the second equation
of \eqref{zd24} is equivalent to
\begin{equation}\begin{cases}\label{zd25}
\begin{aligned}
b^2c_0&=c^2\beta,\\
f&=\lambda b^4.\\
\end{aligned}\end{cases}\end{equation}
Differentiating both sides of the first equation of \eqref{zd25} by
$y^i$ yields
\begin{equation}\label{zd26}
c^2b_i=b^2c_i.
\end{equation}

\textbf{Case I:} $c(x)=const$. We have $c=0$ by \eqref{zd26}. So
$\beta$ is a constant Killing
 form. Thus by Theorem \ref{zza2}, we have $\sigma=\frac{1}{4}\lambda b^2\geq0$.

\textbf{Case II:} $c(x)\neq const$.
 We can rewrite \eqref{zd26} as
\begin{equation*}
(c^{-1})_{\,\,|i}=- \frac{b_i}{b^2}.
\end{equation*}
So we have
\begin{equation*}
(b^2c^{-1})_{\,\,|i}=(b^2)_{\,\,|i}c^{-1}+b^2(c^{-1})_{\,\,|i}=2cb_ic^{-1}+b^2(-
\frac{b_i}{b^2})=b_i,
\end{equation*}
which means that $s_{ij}=0$. Thus $\beta$ is closed.
%Contracting both sides of \eqref{p5} with $b^i$, we get $b^kc_k=c^2$.
From \eqref{zd24}, we get $\sigma=0$.

Note that by Lemma \ref{zzd1}, the last two equations of
\eqref{zd24} always hold.\end{proof}

\vspace{8mm}

\section{Kropina Metrics Through Navigation Description}\label{Z5}

In this section, we will algebraically derive an expression for $F$,
and obtain another characterization of Einstein Kropina metric.

Notice that we restrict our consideration to the domain where
$\beta=b_i(x)y^i>0$, which is equivalent to $W_0=W_i(x)y^i>0$.

 Let $h = \sqrt{ h_{ij}(x)
y^i y^j}$ be a Riemannian metric and $W = W^
i\frac{\partial}{\partial x^i}$ a vector field on $M$. We can
determine the Finsler metric $F=F( x,y)$ as follows
\begin{equation*}
|| \frac{y}{F}- W||_{h}= \sqrt{h_{ij}(x)(\frac{y^i}{F}-
W^i)(\frac{y^j}{F}- W^j)}=1.
\end{equation*}
It is equivalent to
\begin{equation}\label{ze1}
\frac{h^2}{F^2}- 2\frac{W_0}{F}+||W||^2_{h}=1,
\end{equation}
where $W_i:=h_{ij}W^j$ and $W_0:=W_iy^i$.

Let $F=\frac{\alpha^2}{\beta}$. Solving \eqref{ze1} for $h$ and $W$,
we have that
\begin{equation}\label{ze2}
0=h^2\beta^2-2W_0\beta\alpha^2+(||W||^2_{h}-1)\alpha^4.
\end{equation}
Since $h^2\beta^2$ is divisible by $\alpha^2$, we conclude that
$h^2=e^{2\rho}\alpha^2$ for some function $\rho=\rho(x)$ on $M$.
Plugging it into \eqref{ze2} yields
\begin{equation}\label{ze3}
0=(e^{2\rho}\beta-2W_0)\beta+(||W||^2_{h}-1)\alpha^2.
\end{equation}

\eqref{ze3} is equivalent to
\begin{equation}\begin{cases}\begin{aligned}\label{zee3}
\eta&:=e^{2\rho}\beta-2W_0,\\
0&=\eta\beta+(||W||^2_{h}-1)\alpha^2.
\end{aligned}\end{cases}\end{equation}

Now we consider second equation of \eqref{zee3} into two cases: 1)
If $\eta=t\beta$ for some function $t=t(x)$ on $M$, then $t
b_ib_j=(||W||^2_{h}-1)a_{ij}$. By the theory of matrix rank, we know
that $t=||W||^2_{h}-1=0$. So $\eta=0$; 2) If $\eta\neq t\beta$ for
any function $t=t(x)$ on $M$, then we just choose the suitable
direction $y$, such that $\eta(y)=0$. For the positive definiteness
of $\alpha$, $\alpha(y)\neq 0$, so we get $||W||^2_{h}-1=0$. Above
all, $||W||_{h}-1=0$ and $\eta=0.$ So till now, we have
\begin{equation}\label{ze4}
h_{ij}=e^{2\rho}a_{ij},~~~2W_i=e^{2\rho}b_{i}~~~and~~~
e^{2\rho}b^2=4.
\end{equation}

Conversely, assume that $||W||_h= \sqrt{h_{ij}(x)W^i W^j}=1$.
Solving \eqref{ze1} for $F$, we obtain $F=\frac{h^2}{2W_0}$. Let
$\alpha^2=h^2$ and $\beta=2W_0$. Thus $F=\frac{\alpha^2}{\beta}$ is
a Kropina metric.

 Hence, we obtain the following theorem.
\begin{theorem} A Finsler metric $F$ is of Kropina
type if and only if it solves the navigation problem on some
Riemannian manifold $(M,h)$, under the influence of a wind $W$ with
$||W||_{h}=1$. Namely, $F=\frac{\alpha^2}{\beta}$ if and only if
$F=\frac{h^2}{2W_0}$, where $ h^2=e^{2\rho}\alpha^2$,
$2W_0=e^{2\rho}\beta$ and $e^{2\rho}b^2=4$.
\end{theorem}
And we call such a pair $(h,W)$ the navigation data of the Kropina
metric $F$.

{\bf {Remark.}}\, Similar navigation idea for Kropina metrics
appeared in \cite{yo}, where they unnaturally assumed that
$||W||_{h}=1$. As stated in \cite{bao2}, the navigation description
for Randers metrics is guaranteed by the condition $||W||_{h}<1$. In
a sense, the navigation idea for Kropina metrics may be considered
to be the limiting case of Randers metrics, as $||W||_{h}$
approaches to 1.

%%%%              %%%%
%%%%
In order to prove Theorem \ref{zza3}, we first need to reexpress the
Einstein Kropina characterization of Theorem \ref{zza1} in terms of
the navigation data $(h,W)$. To that end, it is helpful to first
relate the covariant derivative $b_{i|j}$ of $b$ (with respect to
$\alpha$) to the covariant derivative $W_{i;j}$ of $W$ (with respect
to $h$).

Let
\begin{equation*}\begin{aligned}\label{dd8}
&\mathcal{R}_{ij}:=\frac{1}{2}(W_{i;j}+
W_{j;i}),\,\,\,\,\,\,\mathcal{S}_{ij}:=\frac{1}{2}(W_{i;j}-
W_{j;i}),\\
&\mathcal{S}^i_{\,\,\,j}:= h^{ik}\mathcal{S}_{kj},\,\,
\mathcal{S}_j:=
W^i\mathcal{S}_{ij},\,\,\mathcal{R}_j:=W^i\mathcal{R}_{ij},\,\,\mathcal{R}:=\mathcal{R}_jW^j,\\
\end{aligned}\end{equation*}
where ";" denotes the covariant differentiation with respect to $h$.

By conformal properties, we have followings
\begin{equation}\begin{aligned}\label{ze5}
r_{ij}&=2e^{-2\rho}(\mathcal{R}_{ij}-W^k\rho_kh_{ij}),\\
\end{aligned}\end{equation}
\begin{equation}\begin{aligned}\label{ze6}
s_{ij}&=2e^{-2\rho}(\mathcal{S}_{ij}+\rho_iW_j-\rho_jW_i),\\
\end{aligned}\end{equation}
where $\rho_i=\frac{\partial \rho}{\partial x^i}$.

%%%%              %%%%
%%%%

\begin{lemma} \label{zze3}$r_{00}=c(x)\alpha^2$ is equivalent to $\mathcal{R}_{ij}=0$. In this case, $W^k\rho_k=-\frac{1}{2}c$.
 \end{lemma}
\begin{proof} Firstly, assume that $r_{00}=c(x)\alpha^2$.
It is equivalent to $r_{ij}=ca_{ij}$. Contracting both sides of it
with $b^ib^j$, we have $r=cb^2$.

By the third equation of \eqref{ze4}, we have
\begin{equation}\label{ze66}
0=b^2\rho_k+r_k+s_k.
\end{equation}
Contracting \eqref{ze66} with $b^k$ yields
\begin{equation*}\begin{aligned}
0&=b^2\rho_kb^k+r=2b^2\rho_kW^k+cb^2.\\\end{aligned}\end{equation*}
 So $W^k\rho_k=-\frac{1}{2}c$.

Then plugging \eqref{ze5} into $r_{ij}=ca_{ij}$, we get
\begin{equation}\begin{aligned}\label{ze7}
ce^{-2\rho}h_{ij}&=2e^{-2\rho}(\mathcal{R}_{ij}-W^k\rho_kh_{ij})\\
&=2e^{-2\rho}(\mathcal{R}_{ij}+\frac{1}{2}ch_{ij}).\\
\end{aligned}\end{equation}
Obviously $\mathcal{R}_{ij}=0$.

Conversely, by $\mathcal{R}_{ij}=0$ and \eqref{ze5}, we have
\begin{equation*}\begin{aligned}
r_{ij}&=-2e^{-2\rho}W^k\rho_kh_{ij}.\\
\end{aligned}\end{equation*}
That is $r_{ij}=ca_{ij}$, where $c=c(x)=-2W^k\rho_k$. This completes
the proof.
 \end{proof}
\vskip 5mm
\begin{theorem}\label{zze4}
Let $F=\frac{\alpha^2}{\beta}$ be a non-Riemannian Kropina metric on
an n-dimensional manifold $M$, $n\geq 2$. Assume the pair $(h,W)$ is
it's navigation data. Then $F$ is an Einstein metric if and only if
$h$ is an Einstein metric and $W$ is a unit Killing vector field. In
this case, $\sigma=\delta\geq0$, where $\delta=\delta (x)$ is the
Einstein scalar of $h$. Moreover, $F$ is Ricci constant for
$n\geq3$.
\end{theorem}

\begin{proof} Now assume that $F=\frac{h^2}{2W_0}$ is an Einstein metric. Then $\frac{h^2}{W_0}$ is also an Einstein Kropina metric.
 By Theorem \ref{zza1}, we have $r_{00}=c(x)\alpha^2$ for $n\geq2$. Then
by Lemma \ref{zze3}, $\mathcal{R}_{ij}=0$ holds. So $W_0$ is a unit
Killing form with respect to $h$. Thus for the Kropina metric
$\frac{h^2}{W_0}$, we know that it is Einstein and $W_0$ is a unit
Killing form. Then according to Theorem \ref{zzd2}, we know $h$ is
also an Einstein metric. Conversely, assume $h$ is an Einstein
metric and $W$ is a unit Killing vector field. Then $W_0$ is a unit
Killing form with respect to $h$. By Theorem \ref{zzd2}, we get
$\frac{h^2}{W_0}$ is an Einstein metric and so is
$F=\frac{h^2}{2W_0}$.

By Theorem \ref{zzd2}, we obtain that the Einstein scalar of
$\frac{h^2}{W_0}$ is
\begin{equation*}
\frac{1}{4}\delta\|W\|^2_h=\frac{1}{4}\delta=-\frac{1}{4}\mathcal{S}^i_{\,\,\,j}\mathcal{S}^j_{\,\,\,i}=\frac{1}{4}\parallel\mathcal{S}_{ij}\parallel^2_h\geq0,
\end{equation*}
where $\delta=\delta(x)$ is the Einstein scalar of $h$. Thus the
Einstein scalar of $F=\frac{h^2}{2W_0}$ is $\sigma=\delta\geq0$. It
completes the proof of theorem.\end{proof} \vspace{3mm}

%%% Examples

By Theorem \ref{zze4}, we can construct a vast Einstein Kropina
metrics by their navigation expressions, i.e., Riemannian Einstein
metrics and unit Killing vector fields. Let $h$ be
$n$-dimensional Riemannian space of constant curvature $\mu$. Denote
$h={||dx||^2}/H^2$, where $H:=1+\frac {\mu}4||x||^2$ and
$||\cdot||^2$ is the standard metric in Euclidean space. Then the
general solutions of Killing vector field $W$ with respect to $h$
are
\begin{equation}\label{ze7}
W_i(x)=\frac 1{H^2}\left\{\sum_jQ_{ij}x^j+c_i-\frac 14\mu
||x||^2c_i+\frac 12\left [\sum_k\mu c_kx^k\right]x^i\right\},
\end{equation}
where $Q_{ij}=-Q_{ji}$ and $c_i$ are $\frac 12n(n+1)$ constants, see \cite{SYB}.
 So there
exist lots of unit Killing vector fields. We list a special case
here.

\begin{example} Let $M$ be an $3$-dimensional unit sphere with standard metric $h$.
Let
\begin{equation*}
Q
 =\left(\begin{array}{ccc}
 0&a&b\\
 -a&0&c\\
  -b&-c&0\\
 \end{array}\right),\qquad (c_1,c_2,c_3)=\pm(c,-b,a),
 \end{equation*}
 where $a^2+b^2+c^2=1$ and $a,b,c$ are all non-zero constants.
Define $W=W^i\frac{\partial}{\partial x^i}$ with the same form as in
\eqref{ze7}, where $W_i=h_{ij}W^j$. Then $||W||_{h}=1$. Define
$F=\frac{h^2}{2W_0}$, where $W_0=W_iy^i$ and $W_0=W_i(x)y^i>0$. Thus
$F$ is an Einstein Kropina metric.
\end{example}

For Ricci flat Kropina metric, we have the following.
\begin{corollary}\label{zze6}
Let $F=\frac{\alpha^2}{\beta}$ be a non-Riemannian Kropina metric on
an $n$-dimensional manifold $M$, $n\geq 2$. If $F$ is Ricci flat,
then $F$ is Berwald. \end{corollary}

\begin{proof} Assume that $F$ is Ricci-flat. By Theorem \ref{zze4},
we have $0=\sigma=\delta=\parallel\mathcal{S}_{ij}\parallel^2_h$,
which means that $W_0$ is closed. Thus $W_0$ is parallel with
respect to $h$. So $G^i=\widetilde{G}^i$, where $\widetilde{G}^i$
denote the geodesic coefficients of $h$. Hence $F$ is a Berwald
metric. It completes the proof of Corollary \ref{zze6}.
\end{proof}

Finsler metrics, which are of constant flag curvature, are special
cases of Einstein metrics. We have following results.
\begin{corollary}\label{zze7}[see\cite{yo1}]
Let $F=\frac{\alpha^2}{\beta}$ be a non-Riemannian Kropina metric on
an $n$-dimensional manifold $M$, $n\geq 2$. $F$ is of constant flag
curvature $K$ if and only if the following conditions hold:\par
(1)$W$ is a unit Killing vector field,\par (2)The Riemannian space $(M,
h)$ is of nonnegative constant curvature $K$.
\end{corollary}
\begin{proof}
Suppose that $F$ is of constant flag curvature $K$, i.e.,
\begin{equation}\label{ze11}
R^i_{\,\,\,k}=K(F^2\delta^i_{\,\,\,k}-g_{ij}y^jy^k).
\end{equation}
Then we have
\begin{equation*}
Ric=\sigma F^2,  \qquad\qquad \sigma:=(n-1)K=const,
\end{equation*}
i.e., $F$ is an Einstein metric. By Theorem \ref{zzd2}, $h$ is an
Einstein metric, $W_0$ is a unit Killing form with respect to $h$
and $\sigma=\delta\geq0$, where $\delta=\delta(x)$ is the Einstein
scalar of $h$. So $K\geq0$.

By a direct computation, we can rewrite \eqref{ze11} as
\begin{equation}\label{ze12}\begin{aligned}
&K\frac{h^4}{4W_0^2}[\delta^i_{\,\,\,k}-\frac{2}{h^2}y^i\tilde{y}_k+\frac{y^iW_k}{W_0}]\\
=&\tilde{R}^i_{\,\,\,k}-\frac{h^2}{W_0}\mathcal{S}^i_{\,\,\,0;k}
+\frac{\tilde{y}_k}{W_0}\mathcal{S}^i_{\,\,\,0;0}
-\frac{h^2}{2W_0^2}\mathcal{S}^i_{\,\,\,0;0}W_k+\frac{h^2}{2W_0}\mathcal{S}^i_{\,\,\,k;0}
+\frac{h^2}{2W_0^2}\mathcal{S}^i_{\,\,\,j}\mathcal{S}^j_{\,\,\,0}\tilde{y}_k
-\frac{h^4}{4W_0^3}\mathcal{S}^i_{\,\,\,j}\mathcal{S}^j_{\,\,\,0}W_k\\
&-\frac{h^4}{4W_0^2}\mathcal{S}^i_{\,\,\,j}\mathcal{S}^j_{\,\,\,k},
\end{aligned}\end{equation}
where $\tilde{y}_k:=h_{ik}y^i$.
 Multiplying both sides of \eqref{ze12}
by $4W_0^3$ yields
\begin{equation}\label{ze13}\begin{aligned}
0=&4W_0^3\tilde{R}^i_{\,\,\,k}-4h^2W_0^2\mathcal{S}^i_{\,\,\,0;k}
+4W_0^2\mathcal{S}^i_{\,\,\,0;0}\tilde{y}_k
-2h^2W_0\mathcal{S}^i_{\,\,\,0;0}W_k
+2h^2W_0^2\mathcal{S}^i_{\,\,\,k;0}
+2h^2W_0\mathcal{S}^i_{\,\,\,j}\mathcal{S}^j_{\,\,\,0}\tilde{y}_k\\
&
-h^4\mathcal{S}^i_{\,\,\,j}\mathcal{S}^j_{\,\,\,0}W_k-h^4W_0\mathcal{S}^i_{\,\,\,j}\mathcal{S}^j_{\,\,\,k}
-K[h^4W_0\delta^i_{\,\,\,k}-2h^2W_0y^i\tilde{y}_k+h^4y^iW_k].
\end{aligned}\end{equation}

For division reason again, we can simplify \eqref{ze13} as
\begin{equation}\label{ze14}\begin{aligned}
0=&4W_0^2\tilde{R}^i_{\,\,\,k}-4h^2W_0\mathcal{S}^i_{\,\,\,0;k}
+4W_0\mathcal{S}^i_{\,\,\,0;0}\tilde{y}_k
-2h^2\mathcal{S}^i_{\,\,\,0;0}W_k +2h^2W_0\mathcal{S}^i_{\,\,\,k;0}
+2Kh^2(W_0W^i-y^i)\tilde{y}_k\\
&-Kh^4(W^iW_k-\delta^i_{\,\,\,k})
-Kh^2(h^2\delta^i_{\,\,\,k}-2y^i\tilde{y}_k)-Kh^4W^iW_k.
\end{aligned}\end{equation}

Contracting  \eqref{ze14} with $\tilde{y}_i$ yields
\begin{equation}\label{ze15}\begin{aligned}
\mathcal{S}_{k0;0}=K(W_0\tilde{y}_k-h^2W_k).
\end{aligned}\end{equation}
From it, we have
\begin{equation}\label{ze16}\begin{aligned}\begin{cases}
&\mathcal{S}^i_{\,\,\,0;k}=-\mathcal{S}^i_{\,\,\,k;0}-K(2W^i\tilde{y}_k-y^iW_k-W_0\delta^i_{\,\,\,k}),\\
&\mathcal{S}^i_{\,\,\,0;0}=K(-h^2W^i+W_0y^i).
\end{cases}\end{aligned}\end{equation}

Plugging \eqref{ze15} and \eqref{ze16} into \eqref{ze14} yields
\begin{equation}\label{ze17}\begin{aligned}
0=&2W_0^2\tilde{R}^i_{\,\,\,k}+3h^2W_0\mathcal{S}^i_{\,\,\,k;0}
+3Kh^2W_0W^i\tilde{y}_k-3Kh^2W_0y^iW_k
-2Kh^2W_0^2\delta^i_{\,\,\,k}+2KW_0^2y^i\tilde{y}_k.
\end{aligned}\end{equation}
The above equation shows that $h^2$ divides
$2W_0^2\tilde{R}^i_{\,\,\,k}
-2Kh^2W_0^2\delta^i_{\,\,\,k}+2KW_0^2y^i\tilde{y}_k=2W_0^2(\tilde{R}^i_{\,\,\,k}
-Kh^2\delta^i_{\,\,\,k}+Ky^i\tilde{y}_k)$. Thus there exists some
function $d^i_{\,\,\,k}=d^i_{\,\,\,k}(x)$ on $M$ such that

\begin{equation}\label{ze18}\begin{aligned}
\tilde{R}^i_{\,\,\,k}
-Kh^2\delta^i_{\,\,\,k}+Ky^i\tilde{y}_k=d^i_{\,\,\,k}h^2.
\end{aligned}\end{equation}
Contracting  \eqref{ze18} with $y^k$ yields $d^i_{\,\,\,k}=0$. Hence
\eqref{ze18} can be simplified as
$\tilde{R}^i_{\,\,\,k}=K(h^2\delta^i_{\,\,\,k}-y^i\tilde{y}_k)$,
which means that $h$ is of constant curvature $K$.

Converse is obvious.
\end{proof}
{\bf {Remark.}}\, R. Yoshikawa, etc., also studied Kropina metrics
of constant flag curvature in terms of $(h,W)$. Their computation is
tedious. Corollary \ref{zze7} is the revised version of Theorem 4 of
\cite{yo1}, which does not restrict nonnegative constant curvature
$K$.

{\bf {Remark.}}\, R. Yoshikawa, etc., also studied Kropina metrics
of constant flag curvature in terms of $(h,W)$. Their computation is
tedious. Corollary \ref{zze7} is the revised version of Theorem 4 of
\cite{yo1}, which does not restrict nonnegative constant curvature
$K$.

 \vspace{8mm}

\section{ S-curvature} \label{Z6}
Let $(M,F)$ be an $n$-dimensional positive definite Finsler space,
$n\geq3$. Let  $\{e_i\}^n_{ i=1}$ be an arbitrary basis for $T_xM$
and $\{\theta^i\}^n _{i=1}$ the dual basis for $T^*_xM$. The
Busemann-Hausdorff volume form is defined by
\begin{equation*} dV_F := \sigma_F \theta^1\wedge ...\wedge
\theta^n,
 \end{equation*}
 where
 \begin{equation*}
\sigma_F := \frac{Vol(B^n(1))}{Vol\{(y^i)\in R^n|F(y^ie_i) < 1\}},
 \end{equation*}
 $Vol$ denotes the Euclidean volume and $Vol(B^n(1))$ denotes the
Euclidean volume of the unit ball in $\Bbb R^n$. The
Busemann-Hausdorff volume form $dV_F $ determines a measure
$\mu_{B-H}$ which is called the Busemann-Hausdorff measure.

Consider a Kropina norm $F=\frac{\alpha^2}{\beta}$ on $M$. We denote
by $dV_F=\sigma_F \theta^1\wedge ...\wedge \theta^n$ and
$dV_{\alpha}=\sigma_{\alpha} \theta^1\wedge ...\wedge \theta^n$ the
volume forms of $F$ and $\alpha$, respectively. Let $\{e_i\}^n_{
i=1}$ be an orthogonal basis for $(T_xM,\alpha)$. Thus
$\sigma_{\alpha}=\sqrt{det(a_{ij})} = 1$. We may assume $\beta=
by^1$. Then
\begin{equation*}
\Omega:= \{(y^i)\in R^n|F(y^ie_i) < 1\}
\end{equation*}
 is a convex body in $R^n$ and
$\sigma_F := \frac{Vol(B^n(1))}{Vol(\Omega)}$. $\Omega$ is given by
\begin{equation*}
\{\frac{2}{b}(y^1-\frac{b}{2})\}^2+\sum^n_{\alpha=2}(\frac{2}{b}y^{\alpha})^2<1.
 \end{equation*}
 Consider the following coordinate transformation $\psi:(y^i
)\rightarrow (u^i )$
\begin{equation*}
u^1:=\frac{2}{b}(y^1-\frac{b}{2}),~~~u^{\alpha}:=\frac{2}{b}y^{\alpha}.
\end{equation*}
$\psi$ sends $\Omega$ onto the unit ball $B^n(1)$ and the Jacobian
of $\psi:(y^i )\rightarrow (u^i )$ is $(\frac{2}{b})^n$. Then
\begin{equation*}
Vol(B^n(1))=\int_{B^n(1)}du^1\ldots
du^n=\int_{\Omega}(\frac{2}{b})^ndy^1\ldots
dy^n=(\frac{2}{b})^nVol(\Omega).
 \end{equation*}
Thus
\begin{equation*}
\sigma_F= \frac{Vol(B^n(1))}{Vol(\Omega)}=(\frac{2}{b})^n.
 \end{equation*}

Hence for a general basis $\{e_i\}^n_{ i=1}$ , we have
\begin{equation*}\label{v}
\sigma_F
:=(\frac{2}{b})^n\sigma_{\alpha},~~~~\sigma_{\alpha}=\sqrt{det(a_{ij})}.
 \end{equation*}
Therefore
\begin{equation*}
dV_F :=(\frac{2}{b})^ndV_{\alpha}.
 \end{equation*}

Take an arbitrary standard local coordinate system $(x^i ,y^i )$.
For a non-zero vector $y\in T_xM$, the distortion $\tau=\tau(x, y)$
is defined by
\begin{equation*}
\tau:=\ln\frac{\sqrt{g_{ij}(x,y)}}{\sigma_F(x)}.
 \end{equation*}
$F$ is Riemannian if and only if $\tau=$constant. In general, $\tau$
is not a constant. However, it can be constant along any geodesic,
but the Finsler metric is not Riemannian. Therefore, it is natural
to study the rate of change of the distortion along geodesics. For a
vector $y\in T_xM \backslash\{0\}$, let $c(t)$ be the geodesic with
$c(0)=x$ and $\dot{c}(0)=y$. The $S$-curvature $S$ is defined by
\begin{equation*}
S(x, y):=\frac{d}{dt}[\tau(c(t),\dot{c}(t))]\mid_{t=0}.
 \end{equation*}
We can rewrite it as
\begin{equation}\label{zf1}
S(x, y)=\frac{\partial G^m}{\partial y^m}-y^m\frac{\partial
\ln\sigma_F }{\partial x^m}.
 \end{equation}

In this section we discuss the $S$-curvature with respect to the
Busemann-Hausdorff volume measure $\mu_{B-H}$.
\begin{proposition}\label{zzf1}
For the Kropina metric $F=\frac{\alpha^2}{\beta}$, we have
\begin{equation}\label{zf2}
S(x, y)=\frac{n+1}{b^2}(r_0-\frac{\beta}{\alpha^2}r_{00}).
 \end{equation}\end{proposition}
\begin{proof}By Proposition \ref{zzc1}, we have
\begin{equation}\label{zf3}
\begin{aligned}
\frac{\partial G^m}{\partial y^m}=&\frac{\partial
\bar{G}^m}{\partial
y^m}-\frac{n}{b^2}s_0+\frac{1}{b^2}r_0-\frac{n+1}{b^2\alpha^2}\beta
r_{00}.
 \end{aligned}\end{equation}

It is known that $S$-curvature of every Riemannian metric vanishes,
i.e.,
\begin{equation}\label{zf4}
0=\frac{\partial \overline{G}^m}{\partial y^m}-y^m\frac{\partial
\ln\sigma_{\alpha} }{\partial x^m}.
 \end{equation}
So plugging \eqref{zf3} and \eqref{zf4} into \eqref{zf1}, we get
\eqref{zf2}. This proves the proposition.\end{proof}

\begin{theorem} \label{zza4}Every Einstein Kropina metric
$F=\frac{\alpha^2}{\beta}$ has vanishing $S$-curvature.\end{theorem}
\begin{proof} Assume that $F$ is an
Einstein metric. By Theorem \ref{zza1}, we have $r_{00}=c\alpha^2$
for some scalar function $c=c(x)$ on $M$. Thus $r_{0}=c\beta$.
Plugging those into \eqref{zf2}, we obtain $S=0$. \end{proof}

\vspace{8mm}

\section{conformal rigidity}\label{Z7}
In this section, we obtain a conformal rigidity result for Einstein
Kropina
 metrics.
\begin{theorem}\label{zzh5}
Any conformal map between Einstein Kropina spaces must be homothetic.
\end{theorem}
\begin{proof}

Let $F=\alpha^2/\beta $, $\widetilde{F}=\phi^{-1}F$ and
$\widetilde{F}=\widetilde{\alpha}^2/\widetilde{\beta}$. Then
$\widetilde{a}_{ij}=\phi^{-2}a_{ij}$ and
$\widetilde{b_{i}}=\phi^{-1}b_{i}$ hold. Let $(h,W)$ and
$(\tilde{h},\widetilde{W})$ be the navigation
 data of $F$ and $\widetilde{F}$, respectively. Suppose that
 $\tilde{h}_{ij}=e^{2\tilde{\rho}}\tilde{a}_{ij}$ and $h_{ij}=e^{2\rho}a_{ij}$ hold.
 So we have
\begin{equation}\begin{cases}\label{zh1}
\begin{aligned}
&\tilde{b}^2=\tilde{a}^{ij}\tilde{b}_i\tilde{b}_j=a^{ij}b_ib_j=b^2,\\
&\tilde{h}_{ij}=e^{2\tilde{\rho}}\tilde{a}_{ij}=e^{2\tilde{\rho}}\phi^{-2}a_{ij}=e^{2(\tilde{\rho}-\rho)}\phi^{-2}h_{ij},\\
&2\widetilde{W}_i=e^{2\tilde{\rho}}\tilde{b}_{i}=e^{2\tilde{\rho}}\phi^{-1}b_{i}=2e^{2(\tilde{\rho}-\rho)}\phi^{-1}W_{i},\\
\end{aligned}\end{cases}
\end{equation}

From \eqref{ze4} and the first equation of \eqref{zh1}, we get that
$\tilde{\rho}=\rho$. So the last two equations of \eqref{zh1} can be
simplified as
\begin{equation}\begin{cases}\label{zh2}
\begin{aligned}
&\tilde{h}_{ij}=\phi^{-2}h_{ij},\\
&\widetilde{W}_i=\phi^{-1}W_{i},\\
\end{aligned}\end{cases}
\end{equation}
which means that two Riemannian metrics $h$ and $\tilde{h}$ are
conformal equivalent.

Firstly by conformal properties, we know that
\begin{equation*}
\tilde{\gamma}^i_{jk}=\gamma^i_{jk}-\phi^{-1}\delta^i_j\phi_k-\phi^{-1}\delta^i_k\phi_j+\phi^{-1}\phi^ih_{jk},
\end{equation*}
where $\gamma^i_{\,\,\,jk}$ and $\tilde{\gamma} ^i_{\,\,\,jk}$ are
the coefficients of Levi-Civita connections of $h$ and $\tilde{h}$,
respectively, $\phi_k:=\frac{\partial \phi}{\partial x^k}$ and
$\phi^k:=h^{ik}\phi_i$.

Let $";"$ and $ "," $ denote the covariant differentiation with
respect to $h$ and $\tilde{h}$, respectively. Thus we have
\begin{equation*}
\begin{aligned}
\widetilde{W}_{j,k}&=\frac{\partial \widetilde{W}_j}{\partial
x^k}-\widetilde{W}_i\tilde{\gamma}^i_{jk}=\phi^{-1}W_{j;k}+\phi^{-2}\phi_jW_k-\phi^{-2}W_i\phi^ih_{jk}.\\
\end{aligned}
\end{equation*}
Hence\begin{equation}\label{zh3}
\widetilde{W}_{j,k}+\widetilde{W}_{k,j}=\phi^{-1}(W_{j;k}+W_{k;j})+\phi^{-2}(W_j\phi_k+W_k\phi_j)
-2\phi^{-2}W^i\phi_ih_{jk}.\end{equation}

Assume that $F$ and $\tilde{F}$ are both Einstein metrics. Thus by
Theorem \ref{zza3}, we know that $W$ and $\widetilde{W}$ are both
constant Killing vector fields. That is $0=W_{j;k}+W_{k;j}$ and
$0=\widetilde{W}_{j,k}+\widetilde{W}_{k,j}$. Hence \eqref{zh3} can
be rewritten as
\begin{equation}\label{zh4}
0=W_j\phi_k+W_k\phi_j -2W^i\phi_ih_{jk}.\end{equation}

Contracting \eqref{zh4} with $h^{jk}$ yields $W^i\phi_i=0$. Putting
it into \eqref{zh4} gets $0=W_j\phi_k+W_k\phi_j$. Then contacting it
with $W^j$ yields $\phi_k=0$, which means that $\phi$=constant. It
completes proof of Theorem \ref{zza5}.\end{proof}

\vskip 5mm

  Xiaoling Zhang

  Department of Mathematics,

  Zhejiang University,

  Hangzhou 310027, China,

  Email: xlzhang@ymail.com
 \vskip 5mm

  Yibing Shen

  CMS and Dept. of Math.,

  Zhejiang University,

  Hangzhou 310027, China,

  Email: yibingshen$@$zju.edu.cn

 \end{document}